\documentclass[a4paper,10pt]{article}
\usepackage[latin1]{inputenc}
\usepackage[T1]{fontenc}
\usepackage{amsmath}
\usepackage{amsfonts}
\usepackage{amstext}
\usepackage{amsthm}
\usepackage[all]{xy}
\usepackage{geometry}
\usepackage{srcltx}
\usepackage{graphics}
\usepackage{epsfig}
\usepackage{subfigure}
\usepackage{slashed}
\usepackage{color}
\usepackage{amssymb}
\usepackage{srcltx}
\newtheorem{theorem}{Theorem}[section]
\newtheorem{lemma}[theorem]{Lemma}

\newtheorem{cor}[theorem]{Corollary}

\DeclareMathOperator{\im}{im}
\DeclareMathOperator{\ind}{ind}

\DeclareMathOperator{\spa}{span}

\title{The Index Bundle for Selfadjoint Fredholm Operators and Multiparameter Bifurcation for Hamiltonian Systems}
 
\author{Robert Skiba and Nils Waterstraat}

\begin{document}
\date{}
\maketitle


\begin{abstract}
\noindent
The index of a selfadjoint Fredholm operator is zero by the well-known fact that the kernel of a selfadjoint operator is perpendicular to its range. The Fredholm index was generalised to families by Atiyah and J\"anich in the sixties, and it is readily seen that on complex Hilbert spaces this so called index bundle vanishes for families of selfadjoint Fredholm operators as in the case of a single operator. The first aim of this note is to point out that for every real Hilbert space and every compact topological space $X$ there is a family of selfadjoint Fredholm operators parametrised by $X\times S^1$ which has a non-trivial index bundle. 
Further, we use this observation and a family index theorem of Pejsachowicz to study multiparameter bifurcation of homoclinic solutions of Hamiltonian systems, where we generalise a previously known class of examples.

\vskip1truecm
\centerline{\textbf{Keywords: Index Bundle, Fredholm Operators, Homoclinics of Hamiltonian Systems} }
\end{abstract}

\section{Introduction}
A bounded operator $T\in\mathcal{L}(H)$ on a Hilbert space $H$ is called \textit{Fredholm} if it has finite dimensional kernel and cokernel. The difference of the dimensions of these spaces is called the \textit{Fredholm index} of $T$. As the kernel of the adjoint of an operator satisfies $\ker(T^\ast)=\im(T)^\perp$, it follows that $\ind(T)=\dim\ker(T)-\dim\ker(T^\ast)$. Hence, if $T$ is selfadjoint, the index vanishes.\\
Atiyah and J\"anich independently introduced a generalisation of the Fredholm index to families of operators which is called the \textit{index bundle}. For a compact topological space $\Lambda$ and a family $L=\{L_\lambda\}_{\lambda\in\Lambda}\subset\mathcal{L}(H)$ of Fredholm operators on $H$, this is an element of the K-theory group $K(\Lambda)$ if $H$ is a complex Hilbert space, or an element of $KO(\Lambda)$ if $H$ is a real Hilbert space. It is not difficult to see that the index bundle of the family of adjoint operators $L^\ast=\{L^\ast_\lambda\}_{\lambda\in\Lambda}$ satisfies $\ind(L^\ast)=-\ind(L)$. Thus we obtain for a family of selfadjoint operators $2\ind(L)=0$. This, however, does not imply that $\ind(L)$ is trivial in the corresponding $K$-theory group if there are elements of order $2$. If $H$ is a complex Hilbert space, it can still be shown that the index bundle of a family of selfadjoint Fredholm operators is trivial. This fact can be found in \cite[Ex. 3.36]{Booss}, and we also recall it in Section 2 below. The argument, however, strongly uses that $H$ is complex and so it is not clear if the same is true if $H$ is real.\\
A first aim of this note is to point out that the index bundle can very well be non-trivial for families of selfadjoint operators on real Hilbert spaces. To explain our approach, let $H$ be a real Hilbert space. The index bundle for families of Fredholm operators of index $0$ belongs to the reduced $KO$-theory group $\widetilde{KO}(\Lambda)$ which is a sub-group of $KO(\Lambda)$. To have a chance to find selfadjoint families with a non-trivial index bundle, we have to consider spaces $\Lambda$ such that $\widetilde{KO}(\Lambda)$ is non-trivial and has elements of order $2$. Suitable spaces that satisfy these assumptions are of the form $\Lambda=X\times S^1$ for an arbitrary compact topological space $X$ as by the product theorem \cite[Cor. 2.4.8]{KTheoryAtiyah}

\[\widetilde{KO}(X\times S^1)\cong\widetilde{KO}(X)\oplus\widetilde{KO}(S^1)\oplus\widetilde{KO}^{-1}(X)\] 
and $\widetilde{KO}(S^1)\cong\mathbb{Z}_2$.\\
Our first theorem below shows that for every separable real Hilbert space $H$ and compact topological space $X$ there is a family $L$ of selfadjoint Fredholm operators on $H$ parametrised by $X\times S^1$ such that $\ind(L)\neq 0\in\widetilde{KO}(X\times S^1)$. Furthermore, we show that if $\widetilde{KO}(X)$ and $\widetilde{KO}^{-1}(X)$ are trivial, then every family of Fredholm operators of index $0$ can either be deformed to a constant family or to our family $L$. \\
Our interest in the above question comes from multiparameter bifurcation theory. The index bundle has been used to construct bifurcation invariants for families of nonlinear operator equations in real Banach spaces for more than three decades (see, e.g., \cite{JacoboK}, \cite{BartschI}, \cite{FPContemp}, \cite{BartschII}, \cite{FPSeveral}, \cite{JacoboTMNAI}, \cite{JacoboTMNAII}, which is far from being exhaustive). Recently, the authors have computed in \cite{RobertNils} and \cite{RobertNilsII} the index bundle for families of discrete dynamical systems to study bifurcation of their solutions.\\
Often the families of equations in bifurcation theory are gradients of functionals, which makes it necessary to deal with selfadjoint operators. This is the case, e.g., when studying bifurcation of homoclinic solutions of Hamiltonian systems. It is the main objective of this paper to point out that the index bundle can be used to study multiparameter bifurcation in this setting. To the best of our knowledge this has not been done before. The only work we are aware of that has considered a related question is \cite{PortaluriNils}, where the operators were complexified to avoid working with index bundles on a real Hilbert space. Finally, we construct an explicit example of Hamiltonian systems having a non-trivial index bundle, which generalises previous work of Pejsachowicz from \cite{Jacobo}.\\ 
The paper is structured as follows. In the next section we outline the construction of the index bundle and recall the Atiyah-J\"anich Theorem that we need later. Our theorem on the non-triviality of the index bundle is stated and proved in Section 3. Finally, Section 4 is devoted to bifurcation theory. Here we firstly obtain an abstract multiparameter bifurcation result for homoclinics of Hamiltonian systems by using an index theorem of Pejsachowicz. Secondly, we construct an explicit example in this setting with a non-trivial index bundle and compare it to an example from \cite{PortaluriNils}.


\section{The Index Bundle}
The aim of this section is to recall the construction of the index bundle and the Atiyah-J\"anich Theorem. There are several references on this subject, which discuss the constructions on various levels of generality (see, e.g., \cite{KTheoryAtiyah}, \cite{Booss}, \cite{Banach Bundles} and \cite{indbundleIch}). Here we mainly follow \cite{FPContemp}, and assume throughout that $H$ is real, separable and of infinite dimension. Let us recall that the space $\Phi(H)$ of all Fredholm operators on $H$ has the path components 

\[\Phi_k(H)=\{T\in\Phi(H):\,\ind(T)=k\}.\]
As we are eventually interested in families of selfadjoint Fredholm operators, we henceforth assume that all our families are in the path component $\Phi_0(H)$.\\
Let $\Lambda$ be a compact topological space and $L:\Lambda\rightarrow\Phi_0(H)$ a family of Fredholm operators of index $0$. It is not difficult to see that there is a finite dimensional subspace $V\subset H$ such that

\begin{align}\label{transversalBundle}
\im(L_\lambda)+V=H,\quad\lambda\in\Lambda.
\end{align}
If $P$ denotes the orthogonal projection onto the orthogonal complement $V^\perp$ of $V$, then we obtain from \eqref{transversalBundle} that the composition

\[H\xrightarrow{L_\lambda} H\xrightarrow{P}V^\perp\]
is surjective for every $\lambda\in\Lambda$. Consequently, the family of kernels of $PL_\lambda$ canonically are a vector bundle $E(L,V)$ over $\Lambda$ whose total space is given by

\[\{(\lambda,u)\in\Lambda\times H:\, L_\lambda u\in V\}.\] 
It is easy to see that the dimension of $E(L,V)$ is the dimension of $V$. Hence, if we denote by $\Theta(V)$ the product bundle with fibre $V$, we obtain a reduced $KO$-theory class

\[\ind(L)=[E(L,V)]-[\Theta(V)]\in\widetilde{KO}(\Lambda).\]
It can be shown that this definiton does not depend on the choice of $V$, so that $\ind(L)$ indeed is a well-defined element of $\widetilde{KO}(\Lambda)$ that only depends on the family $L$. Several properties of this \textit{index bundle} are perfect analogues of the properties of the integral Fredholm index, e.g.,

\begin{itemize}
 \item[(i)] if $L_\lambda$ is invertible for all $\lambda\in\Lambda$, then $\ind(L)=0\in\widetilde{KO}(\Lambda)$.
 \item[(ii)] if $\mathcal{H}:\Lambda\times I\rightarrow\Phi_0(H)$ is a homotopy of Fredholm operators, then \[\ind(\mathcal{H}_0)=\ind(\mathcal{H}_1)\in\widetilde{KO}(\Lambda).\]
 \item[(iii)] if $L,M:\Lambda\rightarrow\Phi_0(H)$ are two families, then $\ind(LM)=\ind(L)+\ind(M)\in\widetilde{KO}(\Lambda)$.
\end{itemize}
Moreover, the index bundle has the following naturality property

\begin{itemize}
\item[(iv)] if $\Lambda'$ is another compact space and $f:\Lambda'\rightarrow\Lambda$, then
 
 \[\ind(f^\ast L)=f^\ast\ind(L)\in\widetilde{KO}(\Lambda'),\]
 where the family $f^\ast L:\Lambda'\rightarrow\Phi_0(H)$ is defined by $(f^\ast L)_\lambda=L_{f(\lambda)}$.
\end{itemize}
The probably most remarkable property of the index bundle is the following \textit{Atiyah-J\"anich Theorem} \cite[App.]{KTheoryAtiyah}.

\begin{theorem}\label{Atiyah-Jaenich}
The map

\[\ind:[\Lambda,\Phi_0(H)]\rightarrow\widetilde{KO}(\Lambda)\]
is a bijection, where $[\Lambda,\Phi_0(H)]$ denotes the homotopy classes of maps $\Lambda\rightarrow\Phi_0(H)$. 
\end{theorem}
\noindent
Note that $[\Lambda,\Phi_0(H)]$ actually is a semi-group, and (iii) implies that the index map in Theorem \ref{Atiyah-Jaenich} is a semi-group homomorphism.\\
The following lemma is well known, but we include its proof for the convenience of the reader.

\begin{lemma}\label{indexadjoint}
Let $L^\ast$ denote the family of adjoints of the family $L:\Lambda\rightarrow\Phi_0(H)$, i.e. $L^\ast=\{L^\ast_\lambda\}_{\lambda\in I}$. Then

\[\ind(L^\ast)=-\ind(L)\in\widetilde{KO}(\Lambda).\]
\end{lemma}
\begin{proof}
For $\varepsilon>0$, we have

\[\|(L^\ast_\lambda L_\lambda+\varepsilon) u\|^2=\|L^\ast_\lambda L_\lambda u\|^2+\varepsilon\|L_\lambda u\|^2+\varepsilon^2\|u\|^2\geq\varepsilon^2\|u\|^2,\quad u\in H.\]
Consequently, $L^\ast_\lambda L_\lambda+\varepsilon$ is injective and has closed range. As  $L^\ast_\lambda L_\lambda+\varepsilon$ is also selfadjoint, it is an invertible operator. Now the properties (i)-(iii) from above imply

\[0=\ind(L^\ast L+\varepsilon I_H)=\ind(L^\ast L)=\ind(L^\ast)+\ind(L),\]
which shows the assertion.
\end{proof}
\noindent
Let us now consider the case of families of selfadjoint operators. Note that by Lemma \ref{indexadjoint}, $2\ind(L)=0\in\widetilde{KO}(\Lambda)$, which however only implies that $\ind(L)$ is trivial if $\widetilde{KO}(\Lambda)$ has no 2-torsion.\\
Let us point out that the construction of the index bundle works analogously if $H$ is a complex Hilbert space, in which case $\ind(L)$ is an element of the reduced $K$-theory group $\widetilde{K}(\Lambda)$. If, in this case, $L:\Lambda\rightarrow\Phi_S(H)$ is a family of selfadjoint operators, we can consider the homotopy $\mathcal{H}:\Lambda\times I\rightarrow\Phi_0(H)$ given by $\mathcal{H}_{(\lambda,s)}=L_\lambda+ s\,i\,I_H$. As the spectra of selfadjoint operators are real, we see that $\mathcal{H}_{(\lambda,s)}$ is invertible for $s\neq 0$. Hence it follows from the homotopy invariance in (ii) and the property (i) that $\ind(L)=\ind(\mathcal{H}_1)=0\in\widetilde{K}(\Lambda)$. In other words, the index bundle of a family of selfadjoint operators on a complex Hilbert space vanishes (see \cite[Ex. 3.36]{Booss}).\\
The previous argument obviously does not work if $H$ is a real Hilbert space as the homotopy $\mathcal{H}$ could not be defined without complexifying $H$. A uniform gap in the spectrum close to $0$ would allow a similar argument, but the next section shows that such a gap does not exist in general.  


\section{The Theorem}

\subsection{Statement of the Theorem and a Corollary}
The following theorem shows that the index bundle can be non-trivial for families of selfadjoint Fredholm operators on real Hilbert spaces.

\begin{theorem}\label{main}
Let $X$ be a compact topological space and $H$ a real separable Hilbert space of infinite dimension.

\begin{itemize}
\item[(i)] There exists a family $L:X\times S^1\rightarrow\Phi_S(H)$ such that 

\[\ind(L)\neq 0\in\widetilde{KO}(X\times S^1).\]
\item[(ii)] If $\widetilde{KO}(X)$ and $\widetilde{KO}^{-1}(X)$ are trivial, then every family $M:X\times S^1\rightarrow\Phi_0(H)$ is either homotopic to a constant family or to the family $L$ in $\Phi_S(H)$ from (i).
\end{itemize}
\end{theorem}
\noindent
Note that the assumptions of (ii) hold in particular if $X$ is contractible, and for $X=S^n$ if $n\equiv 5\,\text{mod}\, 8$ or $n\equiv 6\,\text{mod}\,8$. The special case that $X$ is a point is worth to be written down as a corollary.

\begin{cor}
Every loop in $\Phi_0(H)$ is homotopic to a loop in $\Phi_S(H)$.
\end{cor}
\noindent
We now prove Theorem \ref{main} in the next section.

\subsection{Proof of Theorem \ref{main}}
We split the proof into several steps.

\subsubsection*{Step 1: Preliminaries} 
Before we construct the family $L$, we need to recall various preliminaries.\\
The \textit{parity} is a $\mathbb{Z}_2$-valued homotopy invariant for paths in $\Phi_0(H)$ that was constructed by Fitzpatrick and Pejsachowicz (see, e.g., \cite{FPContemp}). To briefly recap its construction, let us firstly recall that a path $M:[a,b]\rightarrow GL(H)$ is called a parametrix for $L:[a,b]\rightarrow\Phi_0(H)$ if $M_\Theta L_\Theta=I_H+K_\Theta$, $\Theta\in[a,b]$, for compact operators $K_\Theta$. By using the Bartle-Graves Theorem, it can be shown that parametrices always exist (see \cite[Thm. 2.1]{FPContemp}). Secondly, the Leray-Schauder degree $\deg_{LS}(f,\Omega)$ is a mapping degree for continuous maps $f:\overline{\Omega}\rightarrow E$ where $\Omega$ is a bounded domain in a real Banach space $E$ and $f=I_E-C$ is a perturbation of the identity by a (generally nonlinear) compact map such that $f(u)\neq 0$ for all $u\in\partial\Omega$. When applied to a linear isomorphism $f=I_E-K$, where $K$ is a compact operator, it can be shown that

\begin{align}\label{degLSlin}
\deg_{LS}(I_E-K)=(-1)^m,
\end{align}  
where 

\[m=\sum_{\lambda<0}{m(\lambda)}\]
and $m(\lambda)$ denotes the algebraic multiplicity of $\lambda$ as an eigenvalue of $I_E-K$. Finally, the parity of a path $L:[a,b]\rightarrow\Phi_0(H)$ having invertible endpoints is given by

\begin{align}\label{parity}
\sigma(L,[a,b])=\deg_{LS}(M_bL_b)\,\deg_{LS}(M_aL_a)\in\mathbb{Z}_2=\{\pm 1\},
\end{align}
where $M:[a,b]\rightarrow GL(H)$ is a parametrix for $L$.\\
The parity is invariant under homotopies in $\Phi_0(H)$ that keep the endpoints invertible. Moreover,

\begin{itemize}
 \item[(C)] if $L_c$ is invertible for some $a<c<b$, then
  \[\sigma(L,[a,b])=\sigma(L,[a,c])\,\sigma(L,[c,b]).\]
  \item[(N)] if $L_\Theta$ is invertible for all $\Theta\in[a,b]$, then $\sigma(L,[a,b])=1$. 
\end{itemize}
For closed paths $L:S^1\rightarrow\Phi_0(H)$ the parity can be defined as in \eqref{parity} as long as there is some $z_0\in S^1$ for which $L_{z_0}$ is invertible. The latter assumption can be lifted by considering $S^1$ as obtained from an interval $[a,b]$ by identifying $a$ and $b$ \cite{FPSeveral}. Then there exists a parametrix $M:[a,b]\rightarrow GL(H)$ and the parity of the closed path $L$ is defined by

\[\sigma(L,S^1)=\deg_{LS}(M_aM^{-1}_b).\] 
It can be shown that this definition does not depend on the choice of the interval $[a,b]$, and that $\sigma(L,S^1)=\sigma(L,[a,b])$ in case that $L_a=L_b$ is invertible. Moreover, the following property is a fundamental relation between the parity and the index bundle (see \cite[Prop. 2.7]{FPSeveral}):

\begin{align}\label{w1indbundle}
\sigma(L,S^1)=w_1(\ind(L))\in\mathbb{Z}_2,
\end{align}
where $w_1:\widetilde{KO}(S^1)\rightarrow\mathbb{Z}_2$ is the isomorphism induced by the first Stiefel-Whitney class. Let us recall that the Stiefel-Whitney classes are maps $w_k:\widetilde{KO}(\Lambda)\rightarrow H^k(\Lambda;\mathbb{Z}_2)$, $k\in\mathbb{N}$, for any compact topological space $\Lambda$, such that

\begin{align}\label{StiefelWhitney}
f^\ast w_k([E]-[F])=w_k(f^\ast([E]-[F]))\in H^k(\widetilde{\Lambda};\mathbb{Z}_2)
\end{align}
for any compact topological space $\widetilde{\Lambda}$ and any continuous map $f:\widetilde{\Lambda}\rightarrow\Lambda$. Note that we identify in \eqref{w1indbundle} the cohomology group $H^1(S^1;\mathbb{Z}_2)$ and $\mathbb{Z}_2$.

\subsubsection*{Step 2: The path $\widetilde{L}^1$}
Let $\{e_k\}_{k\in\mathbb{Z}}$ be a complete orthonormal system of $H$. We denote by $P_0$ the orthogonal projection onto the span of $e_0$ and by $P_\pm$ the orthogonal projections onto the closures of the spans of $\{e_k\}_{\pm k\in\mathbb{N}}$, respectively. Consider the paths of operators

\[\widetilde{L}^1_\Theta=P_+-P_-+\Theta\, P_0,\quad \Theta\in[-1,1],\]
and the constant path $M_\Theta=P_+-P_-+P_0\in GL(H)$. Then

\[M_\Theta\widetilde{L}^1_\Theta=I_H-(1-\Theta)P_0,\quad\Theta\in[-1,1],\]
and thus $M$ is a parametrix for $\widetilde{L}^1$. Note that by \eqref{degLSlin}, $\deg_{LS}(M\widetilde{L}^1_1)=\deg_{LS}(I_H)=1$ and $\deg_{LS}(M\widetilde{L}^1_{-1})=\deg_{LS}(I_H-2P_0)=-1$, which means that $\sigma(\widetilde{L}^1,[-1,1])=-1$.

\subsubsection*{Step 3: The path $\widetilde{L}^2$}
We claim that there is a path $\widetilde{L}^2$ in $\Phi_S(H)\cap GL(H)$ such that $\widetilde{L}^2_0=\widetilde{L}^1_1$ and $\widetilde{L}^2_1=\widetilde{L}^1_{-1}$. Note that $\widetilde{L}^1_1=P_++P_0-P_-$ and $\widetilde{L}^1_{-1}=P_+-P_--P_0$. Let $A_+:\im(P_++P_0)\rightarrow\im(P_+)$ be the right shift, and $A_-:\im(P_-)\rightarrow\im(P_-+P_0)$ the left shift. These are orthogonal operators. We set

\[N:=A_+(P_++P_0)-A_-P_-\in GL(H)\]
and obtain for $u,v\in H$

\begin{align*}
\langle N^\ast\widetilde{L}^1_{-1}Nu,v\rangle&=\langle \widetilde{L}^1_{-1}Nu,Nv\rangle\\
&=\langle A_+(P_++P_0)u,A_+(P_++P_0)v\rangle-\langle A_-P_-u,A_-P_-v\rangle\\
&=\langle (P_++P_0)u,(P_++P_0)v\rangle-\langle P_-u,P_-v\rangle\\
&=\langle(P_++P_0-P_-)u,v\rangle=\langle \widetilde{L}^1_{1}u,v\rangle.
\end{align*}
Consequently, $N^\ast\widetilde{L}^1_{-1}N=\widetilde{L}^1_{1}$. As $GL(H)$ is path-connected, there is a path $\{W_\lambda\}_{\lambda\in I}$ in $GL(H)$ connecting $N$ to the identity $I_H$. Now $\widetilde{L}^2_\lambda:=W^\ast_\lambda \widetilde{L}^1_{-1} W_\lambda$, $\lambda\in I$, is a path in $\Phi_S(H)\cap GL(H)$ that connects the endpoints of $\widetilde{L}^1$.

\subsubsection*{Step 4: Construction of the family $L$}
It readily follows from (C) and (N) that the concatenation of $\widetilde{L}^1$ and $\widetilde{L}^2$ is a closed path $\widetilde{L}:S^1\rightarrow\Phi_S(H)$ such that $\sigma(\widetilde{L},S^1)=-1$.\\
We now define the family $L:X\times S^1\rightarrow\Phi_S(H)$ as the trivial extension $L_{(x,z)}=\widetilde{L}_z$, $(x,z)\in X\times S^1$. Let $\iota:S^1\rightarrow X\times S^1$ be the map $\iota(z)=(x_0,z)$ for some $x_0\in X$. Then $L\circ\iota=\widetilde{L}$ and we get from the naturality (iv) of the index bundle that

\[\iota^\ast\ind(L)=\ind(\iota^\ast L)=\ind(\widetilde{L})\in\widetilde{KO}(S^1).\]
Hence it follows from the naturality of the first Stiefel-Whitney class that

\[\iota^\ast(w_1(\ind(L)))=w_1(\iota^\ast\ind(L))=w_1(\ind(\widetilde{L}))=\sigma(\widetilde{L},S^1),\]
where we have used \eqref{w1indbundle} in the last equality. Finally, as $\sigma(\widetilde{L},S^1)=-1$, this implies that $w_1(\ind(L))\neq 1\in H^1(X\times S^1;\mathbb{Z}_2)$ and so $\ind(L)$ is non-trivial as claimed.

\subsubsection*{Step 5: Proof of (ii)}
Let us first recall from the introduction that

\[\widetilde{KO}(X\times S^1)\cong\widetilde{KO}(X)\oplus\widetilde{KO}(S^1)\oplus\widetilde{KO}^{-1}(X),\]
which yields $\widetilde{KO}(X\times S^1)\cong\widetilde{KO}(S^1)\cong\mathbb{Z}_2$ by the assumptions of (ii). The Atiyah-J\"anich Theorem \ref{Atiyah-Jaenich} now implies that

\[\ind:[X\times S^1,\Phi_0(H)]\rightarrow\widetilde{KO}(X\times S^1)\cong\mathbb{Z}_2\]
is a bijection. Hence there are only two homotopy classes and one of them contains the constant family given by the identity $I_H$. As the family $L$ from (i) has a non-trivial index bundle, it cannot be homotopic to a constant family. Thus any given family in $\Phi_0(H)$ is either homotopic to the constant family $I_H$ or to $L$, which shows the assertion.


\section{Multiparameter Bifurcation for Homoclinic Solutions of Hamiltonian Systems}

\subsection{The Fitzpatrick-Pejsachowicz Bifurcation Theorem and its Limits}
Let $X,Y$ be Banach spaces and $\Lambda$ a compact connected CW-complex. Let $F:\Lambda\times X\rightarrow Y$ be a continuous family of $C^1$-Fredholm maps such that $F(\lambda,0)=0$ for all $\lambda\in\Lambda$. We call $\lambda^\ast\in\Lambda$ a \textit{bifurcation point} if in every neighbourhood of $(\lambda^\ast,0)\in\Lambda\times X$ there is some $(\lambda,u)$ such that $F(\lambda,u)=0$ and $u\neq 0$. If we denote by $L_\lambda:=D_0F_\lambda$ the Fr\'echet derivative of $F_\lambda$ at $0\in X$, then it clearly follows from the implicit function theorem that $L_{\lambda^\ast}$ is not invertible if $\lambda^\ast$ is a bifurcation point. On the other hand, it is readily seen that the non-invertibility of $L_{\lambda^\ast}$ is in general not sufficient for the existence of bifurcation points.\\
We have recalled in the previous section that a non-trivial index bundle implies the existence of some $\lambda^\ast\in\Lambda$ such that $L_{\lambda^\ast}$ is not invertible. Pejsachowicz and Fitzpatrick pointed out the importance of the index bundle for bifurcation theory in a series of papers (cf. e.g. \cite{JacoboK}, \cite{FPContemp}, \cite{JacoboTMNAI}). Here we use the main theorems of \cite{JacoboTMNAI} and \cite{FPSeveral} (see also \cite{JacoboTMNAII}) in a slightly modified version from \cite{NilsBif}. In what follows, we denote by $w_k$, $k\in\mathbb{N}$, the Stiefel-Whitney classes, which map $\widetilde{KO}(\Lambda)$ to $H^k(\Lambda;\mathbb{Z}_2)$. Further, $B(F)\subset\Lambda$ is the set of all bifurcation points in $\Lambda$.

\begin{theorem}\label{PejsachowiczI}
If there is some $\lambda_0\in\Lambda$ for which $L_{\lambda_0}$ is invertible and $w_k(\ind L)\neq 0\in H^k(\Lambda;\mathbb{Z}_2)$ for some $k\in\mathbb{N}$, then $B(F)\neq\emptyset$. Moreover, if $\Lambda$ is a topological manifold of dimension $m\geq 2$ and $1\leq k\leq m-1$, then the dimension of $B(F)$ is at least $m-k$ and the set $B(F)$ is not contractible to a point.
\end{theorem}
\noindent
Here we refer by dimension of the set $B(F)$ to the Lebesgue covering dimension. Note that Pejsachowicz showed in \cite[Rem. 1.2.1]{JacoboTMNAI} that the existence of some $\lambda_0\in\Lambda$ for which $L_{\lambda_0}$ is invertible cannot be lifted.\\
It is a pretty common setting in applications of bifurcation theory that $X=Y$ is a Hilbert space $H$ and the maps $F_\lambda:H\rightarrow H$ are gradients of functionals, i.e., there is a family of $C^2$ functionals $f:\Lambda\times H\rightarrow\mathbb{R}$ such that $F_\lambda=\nabla f_\lambda$ for all $\lambda\in\Lambda$. In this case, $L_\lambda:=D_0F_\lambda$ is a selfadjoint operator. Consequently, $\ind(L)\in\widetilde{KO}(\Lambda)$ is the index bundle of a family of selfadjoint Fredholm operators. It is the main aim of this paper to stress out that Theorem \ref{PejsachowiczI} can be applied to find bifurcation points for important equations that are gradients of functionals.

\subsection{Bifurcation of Homoclinic Solutions of Hamiltonian Systems}
Let $\Lambda$ be a connected closed smooth manifold, and $\mathcal{H}:\Lambda\times\mathbb{R}\times\mathbb{R}^{2n}\rightarrow\mathbb{R}$ a smooth map. We consider the family of Hamiltonian systems

\begin{equation}\label{Hamiltoniannonlin}
\left\{
\begin{aligned}
Ju'(t)+\nabla_u \mathcal{H}_\lambda(t,u(t))&=0,\quad t\in\mathbb{R}\\
\lim_{t\rightarrow\pm\infty}u(t)&=0,
\end{aligned}
\right.
\end{equation}
where $\lambda\in\Lambda$ and

\begin{align}\label{J}
J=\begin{pmatrix}
0&-I_n\\
I_n&0
\end{pmatrix}
\end{align}
is the standard symplectic matrix. In what follows, we assume that $\mathcal{H}$ is of the form

\begin{align}\label{Hgrowth}
\mathcal{H}_\lambda(t,u)=\frac{1}{2}\langle A(\lambda,t)u,u\rangle+G(\lambda,t,u),
\end{align}
where $A:\Lambda\times\mathbb{R}\rightarrow\mathcal{L}(\mathbb{R}^{2n})$ is a family of symmetric matrices, $G(\lambda,t,u)$ vanishes up to second order at $u=0$, and there are $p>0$, $C\geq 0$ and $g\in H^1(\mathbb{R},\mathbb{R})$ such that

\[|D^2_uG(\lambda,t,u)|\leq g(t)+C|u|^p.\]
Moreover, we suppose that $A_\lambda:=A(\lambda,\cdot):\mathbb{R}\rightarrow\mathcal{L}(\mathbb{R}^{2n})$ converges uniformly in $\lambda$ to families 

\begin{align}\label{limits}
A_\lambda(+\infty):=\lim_{t\rightarrow\infty}A_\lambda(t),\quad A_\lambda(-\infty):=\lim_{t\rightarrow-\infty}A_\lambda(t),\quad\lambda\in \Lambda,
\end{align}
and that the matrices $JA_\lambda(\pm\infty)$ are hyperbolic, i.e. they have no eigenvalues on the imaginary axis. Note that by \eqref{Hgrowth}, $\nabla_u \mathcal{H}_\lambda(t,0)=0$ for all $(\lambda,t)\in \Lambda\times\mathbb{R}$, so that $u\equiv 0$ is a solution of \eqref{Hamiltoniannonlin} for all $\lambda\in\Lambda$.\\
Let us now briefly recall the variational formulation of the equations \eqref{Hamiltoniannonlin} from \cite[\S 4]{Jacobo}. The bilinear form $b(u,v)=\langle J u',v\rangle_{L^2(\mathbb{R},\mathbb{R}^{2n})}$, $u,v\in H^1(\mathbb{R},\mathbb{R}^{2n})$, extends to a bounded form on the well known fractional Sobolev space $H^\frac{1}{2}(\mathbb{R},\mathbb{R}^{2n})$, which can be described in terms of Fourier transforms (cf. eg. \cite[\S 10]{Stuart}). Under the assumption \eqref{Hgrowth}, the map $f:\Lambda\times H^\frac{1}{2}(\mathbb{R},\mathbb{R}^{2n})\rightarrow\mathbb{R}$ given by

\[f_\lambda:H^\frac{1}{2}(\mathbb{R},\mathbb{R}^{2n})\rightarrow\mathbb{R},\quad f_\lambda(u)=\frac{1}{2}b(u,u)+\frac{1}{2}\int^\infty_{-\infty}{\langle A(\lambda,t)u(t),u(t)\rangle\,dt}+\int^\infty_{-\infty}{G(\lambda,t,u(t))\,dt}\]
is $C^2$. Moreover, its critical points are the (classical) solutions of \eqref{Hamiltoniannonlin}. Finally,
the second derivative of $f_\lambda$ at the critical point $0\in H^\frac{1}{2}(\mathbb{R},\mathbb{R}^{2n})$ is given by 

\begin{align}\label{L}
D^2_0f_\lambda(u,v)=b(u,v)+\int^\infty_{-\infty}{\langle A(\lambda,t)u(t),v(t)\rangle\,dt}
\end{align}
and, by using the hyperbolicity of $J A_\lambda(\pm\infty)$, it can be shown that the corresponding Riesz representations $L_\lambda:H^\frac{1}{2}(\mathbb{R},\mathbb{R}^{2n})\rightarrow H^\frac{1}{2}(\mathbb{R},\mathbb{R}^{2n})$ are Fredholm. Consequently, the operators $L_\lambda$ are selfadjoint Fredholm operators, and it follows by elliptic regularity that the kernel of $L_\lambda$ consists of the classical solutions of the linear differential equation

\begin{equation}\label{Hamiltonianlin}
\left\{
\begin{aligned}
Ju'(t)+A(\lambda,t)u(t)&=0,\quad t\in\mathbb{R}\\
\lim_{t\rightarrow\pm\infty}u(t)&=0.
\end{aligned}
\right.
\end{equation}
The stable and the unstable subspaces of \eqref{Hamiltonianlin} are

\begin{align*}
E^s(\lambda,0)&=\{u(0)\in\mathbb{R}^2:\,Ju'(t)+A(\lambda,t)u(t)=0,\, t\in\mathbb{R}; u(t)\rightarrow 0, t\rightarrow\infty\},\\ 
E^u(\lambda,0)&=\{u(0)\in\mathbb{R}^2:\,Ju'(t)+A(\lambda,t)u(t)=0,\, t\in\mathbb{R}; u(t)\rightarrow 0, t\rightarrow-\infty\},
\end{align*}
and it is clear that \eqref{Hamiltonianlin} has a non-trivial solution if and only if $E^s(\lambda,0)$ and $E^u(\lambda,0)$ intersect non-trivially. If we consider the systems \eqref{Hamiltoniannonlin} for the limits \eqref{limits}, i.e.,

\begin{equation}\label{Hamiltonianlinlim}
\left\{
\begin{aligned}
Ju'(t)+A(\lambda,\pm\infty)u(t)&=0,\quad t\in\mathbb{R}\\
\lim_{t\rightarrow\pm\infty}u(t)&=0,
\end{aligned}
\right.
\end{equation}
then the corresponding stable and unstable spaces are given by the generalised eigenvectors of $JA(\lambda,\pm\infty)$ with respect to eigenvalues having negative or positive real parts, respectively. These spaces form vector bundles $E^s(\pm\infty)$ and $E^u(\pm\infty)$ over $\Lambda$ such that

\[E^u(+\infty)\oplus E^s(+\infty)\cong E^u(-\infty)\oplus E^s(-\infty)\cong\Theta(\mathbb{R}^{2n}),\] 
and it follows from Pejsachowicz' index formula \cite[Prop. 5.2]{Jacobohomoclinicsfamily} that

\begin{align}\label{ind}
\ind(L)=[E^s(+\infty)]-[E^s(-\infty)]\in\widetilde{KO}(\Lambda).
\end{align}
\noindent
If $B\subset\Lambda$ denotes the set of all bifurcation points of \eqref{Hamiltoniannonlin}, then we obtain from Theorem \ref{PejsachowiczI}  the following result.

\begin{theorem}\label{thm:Stiefel-Whitney}
If there is some $\lambda_0\in\Lambda$ for which \eqref{Hamiltonianlin} only has the trivial solution, and some $k\in\mathbb{N}$ such that

\begin{align}\label{Stiefel-Whitney}
w_k(E^s(+\infty))\neq w_k(E^s(-\infty)),
\end{align}
then $B\neq\emptyset$. If $m=\dim(\Lambda)\geq 2$ and $1\leq k\leq m-1$, then the covering dimension of $B$ is at least $m-k$ and $B$ is not a contractible topological space.
\end{theorem}
\noindent
By Theorem \ref{main}, \eqref{ind} has a chance to be non-trivial, but it is now important to ensure that \eqref{Stiefel-Whitney} can indeed occur in the setting considered in this section, which is our next aim.

\subsection{An Example}
Let $X$ be a connected closed smooth manifold of dimension $m$. We consider for $\Lambda=X\times S^1$ the family \eqref{Hamiltoniannonlin} and assume that there is a $p_0\in X$ such that

\begin{align*}
A_{(p_0,z)}(+\infty)=\lim_{t\rightarrow+\infty}A_{(p_0,z)}(t)=
a_+JS_{\Theta},\qquad A_{(p_0,z)}(-\infty)=\lim_{t\rightarrow-\infty} A_{(p_0,z)}(t)=
a_-JS_{0}
\end{align*}
for $z=e^{i\Theta}$, $\Theta\in[-\pi,\pi]$,

\[S_{\Theta}=\begin{pmatrix}
\cos\Theta&\sin\Theta\\
\sin\Theta&-\cos\Theta
\end{pmatrix}\]
and real numbers $a_\pm\neq 0$. Note that this is in line with the assumptions of the previous section as $a_\pm S_{\Theta}$ are hyperbolic for all $\Theta\in[-\pi,\pi]$.\\
We consider the canonical embedding $\iota:S^1\hookrightarrow X\times S^1$, $z\mapsto(p_0,z)$ and note that 

\[E^s_{p_0}(\pm\infty):=\iota^\ast(E^s(\pm\infty))\in \widetilde{KO}(S^1)\]
are the stable bundles of the equations

\begin{align}\label{diffequ}
Ju'(t)+A((p_0,z),\pm\infty)u(t)&=0,\quad t\in\mathbb{R}.
\end{align}
Now $S_{\Theta}$ is the reflection by the line $e_2(\Theta)=(\cos(\frac{\Theta}{2}),\sin(\frac{\Theta}{2}))$, and consequently $\{e_1(\Theta), e_2(\Theta)\}$ for
$e_1(\Theta)=(-\sin(\frac{\Theta}{2}),\cos(\frac{\Theta}{2}))$ is a basis of eigenvectors of $S_\Theta$ with corresponding eigenvalues $-1$ and $1$. Clearly, $E^s_{p_0}(-\infty)$ is a product bundle, where the fibre is spanned by $(1,0)$ if $a_->0$ and by $(0,1)$ if $a_-<0$. For $E^s_{p_0}(+\infty)$, we just note that every solution of \eqref{diffequ} is of the form

\[u(t)=W\begin{pmatrix}
e^{a_+t}&0\\
0&e^{-a_+t}
\end{pmatrix}W^{-1}c,\quad c\in\mathbb{R}^2,\]
where $W=(e_1(\Theta),e_2(\Theta))$. Consequently,

\begin{align*}
E^s_{p_0}(+\infty)&=\{(z,u)\in S^1\times\mathbb{R}^2:\, u\in\spa\{e_1(\Theta)\}\}\quad \text{if}\, a_+<0,\\
E^s_{p_0}(+\infty)&=\{(z,u)\in S^1\times\mathbb{R}^2:\, u\in\spa\{e_2(\Theta)\}\}\quad \text{if}\, a_+>0,
\end{align*}
and in both cases $E^s_{p_0}(+\infty)$ is isomorphic to the M\"obius bundle and thus non-orientable.\\
In summary, we obtain

\begin{align*}
\iota^\ast w_1(E^s(+\infty))&=w_1(\iota^\ast E^s(+\infty))=w_1(E^s_{p_0}(+\infty))\neq  w_1(E^s_{p_0}(-\infty))=w_1(\iota^\ast E^s(-\infty))\\
&=\iota^\ast w_1(E^s(-\infty)),
\end{align*}
and consequently

\[w_1(E^s(+\infty))\neq w_1(E^s(-\infty)).\]
This shows by Theorem \ref{thm:Stiefel-Whitney} that, if there is some $(p,z)\in X\times S^1$ such that \eqref{Hamiltonianlin} only has the trivial solution, then the set $B\subset X\times S^1$ of bifurcation points of \eqref{Hamiltoniannonlin} is not empty. Moreover, if $m=\dim(X)\geq 1$, then $B$ has at least covering dimension $m$ and is not contractible to a point. 

\subsection{An Example of Pejsachowicz}
The aim of this final section is to revisit an example from \cite{PortaluriNils} that originated from the proof of the main theorem of \cite{Jacobo}. We consider the Hamiltonian systems \eqref{Hamiltoniannonlin} for $\Lambda=T^m=S^1\times\cdots\times S^1$ and

\begin{align}\label{Az}
A_\lambda(t)=\begin{cases}
(\arctan t)JS_{\Theta_1+\ldots +\Theta_m},\quad t\geq 0\\
(\arctan t) JS_0,\quad t<0,
\end{cases},
\end{align}
where $\lambda=(e^{i\Theta_1},\ldots,e^{i\Theta_m})\in T^m$, $\Theta_j\in[-\pi,\pi]$, $j=1,\ldots,m$, and $S_\Theta$ is defined as before. In order to apply the bifurcation result of the previous section, we only need to show that 

\begin{equation}\label{HamiltoniannonlinII}
\left\{
\begin{aligned}
Ju'(t)+A_\lambda(t)u(t)&=0,\quad t\in\mathbb{R}\\
\lim_{t\rightarrow\pm\infty}u(t)&=0,
\end{aligned}
\right.
\end{equation}
only has the trivial solution for some $\lambda_0\in T^m$. Let us recall that the stable and unstable spaces of \eqref{HamiltoniannonlinII} are

\begin{align*}
E^s(\lambda,0)&=\{u(0)\in\mathbb{R}^2:\,Ju'(t)+A_\lambda(t)u(t)=0,\, t\in\mathbb{R}; u(t)\rightarrow 0, t\rightarrow\infty\}\\ 
E^u(\lambda,0)&=\{u(0)\in\mathbb{R}^2:\,Ju'(t)+A_\lambda(t)u(t)=0,\, t\in\mathbb{R}; u(t)\rightarrow 0, t\rightarrow-\infty\}.
\end{align*}
The space $\mathbb{R}^2$ is symplectic with respect to the canonical symplectic form. As the matrices $JS_\lambda(t)$ converge uniformly in $\lambda$ to families of hyperbolic matrices for $t\rightarrow\pm\infty$, it can be shown that $E^s(\lambda,0)$ and $E^u(\lambda,0)$ are Lagrangian subspaces of $\mathbb{R}^2$ (see, e.g., \cite[Lemma 4.1]{Homoclinics}). This implies in particular that $E^s(\lambda,0)$ and $E^u(\lambda,0)$ are one-dimensional.\\
Clearly, there is a non-trivial solution of \eqref{HamiltoniannonlinII} if and only if $E^u(\lambda,0)\cap E^s(\lambda,0)\neq \{0\}$. By a direct computation, one verifies that

\begin{align*}
u_-(t)&=\sqrt{t^2+1}e^{-t\arctan(t)}\begin{pmatrix} 1\\0\end{pmatrix},\, t\leq 0,\\ u_+(t)&=\sqrt{t^2+1}e^{-t\arctan(t)}\begin{pmatrix} \cos\left(\frac{\Theta_1+\ldots+\Theta_m}{2}\right)\\ \sin\left(\frac{\Theta_1+\ldots+\Theta_m}{2}\right)\end{pmatrix},\,t\geq 0,
\end{align*}
are solutions of \eqref{HamiltoniannonlinII} on the negative and positive half-line, respectively, and so $u_-(0)\in E^u(\lambda,0)$ and $u_+(0)\in E^s(\lambda,0)$. As $u_+(0)$ and $u_-(0)$ are linearly dependent if and only if the second component of $u_+$ vanishes, we conclude that \eqref{HamiltoniannonlinII} has a non-trivial solution if and only if $\Theta_1+\ldots+\Theta_m\equiv 0\mod 2\pi$. Consequently, for $m\geq 2$ it follows from our previous section that the covering dimension of the set $B$ of all bifurcation points is at least $m-1$, and $B$ is not contractible to a point. This was shown in \cite[Thm. 4.3]{PortaluriNils}. For $m=1$, i.e. if $\Lambda=S^1$, we can conclude that there is a bifurcation point of \eqref{Hamiltoniannonlin}. Actually, as a non-trivial solution of \eqref{HamiltoniannonlinII} is necessary for a bifurcation point, $1=e^0\in S^1$ is the only bifurcation point of \eqref{Hamiltoniannonlin} in this case. This can also be obtained from the main theorem of \cite{Jacobo}.

\vspace*{1.3cm}

Robert Skiba\\
Faculty of Mathematics and Computer Science\\
Nicolaus Copernicus University in Toru\'n\\
Poland\\
E-mail: robert.skiba@mat.umk.pl

\vspace{0.7cm}

Nils Waterstraat\\
Martin-Luther-Universit\"at Halle-Wittenberg\\
Naturwissenschaftliche Fakult\"at II\\
Institut f\"ur Mathematik\\
06099 Halle (Saale)\\
Germany\\
E-mail: nils.waterstraat@mathematik.uni-halle.de

\end{document}